\documentclass[a4paper,11pt]{amsart}

\usepackage{enumerate}

\newtheorem{theorem}{Theorem}
\newtheorem{proposition}{Proposition}[section]

\newtheorem{lemma}[proposition]{Lemma}
\theoremstyle{definition}

\theoremstyle{remark}

\newcommand{\abs}[1]{\lvert#1\rvert}

\newcommand{\R}{\mathbf{R}}
\newcommand{\N}{\mathbf{N}}

\newcommand{\st}{\: :\:}

\newcommand{\weakto}{\rightharpoonup}
\newcommand{\muleb}[1]{\mathcal{L}^{#1}}

\newcommand{\Norm}[1]{\Vert#1\Vert}

\date{\today}
\title[Explicit approximation of symmetrization]{Explicit approximation of the symmetric rearrangement by polarizations}
\author{Jean Van Schaftingen}
\thanks{The author was supported by the Fonds sp\'eciaux de recherche (Universit\'e catholique de Louvain)}
\address{Universit\'e catholique de Louvain\\ D\'epartement de Math\'ematique\\
Chemin du Cyclotron, 2 \\ 1348 Louvain-la-Neuve\\ Belgium}
\email{Jean.VanSchaftingen@uclouvain.be}


\begin{document}

\maketitle

\begin{abstract}
We give an explicit sequence of polarizations such that for every measurable function, the sequence of iterated polarizations converge to the symmetric rearrangement of the initial function.
\end{abstract}

\section{Introduction}

The symmetric rearrangement is a tool used in the study of isoperimetric inequalities and symmetry of the solution of variationnal problem. Given $u : \R^N \to \R^+ \cup \{+\infty\}$, the symmetric rearrangement, or Schwarz symmetrization, $u^* : \R^N \to \R^+ \cup \{+\infty\}$ is the unique function such that for every $\lambda > 0$, there exists $R \ge 0$ such that 
\[
 \{ x \in \R^N \st u^*(x) > \lambda\}=B(0,R),
\]
and
\[
 \muleb{N} \{ x \in \R^N \st u^*(x) > \lambda\}=\muleb{N}\{ x \in \R^N \st u(x) > \lambda\}.
\]
The function $u^*$ is thus a radial and radially decreasing function whose sublevel sets have the same measure as those of $u$.

Since rearrangement preserves the measure of sublevel sets, if $u \in L^p(\R^N)$, then $u^* \in L^p(\R^N)$ and
\begin{equation}
\label{eqCavalieri}
 \int_{\R^N} \abs{u^*}^p =\int_{\R^N} \abs{u}^p.
\end{equation}
Rearrangements also brings the mass of a function around the origin: if $u \in L^p(\R^N)$ and $v \in L^{\frac{p}{p-1}}(\R^N)$, 
\begin{equation}
 \label{eqHardyLittlewood}
 \int_{\R^N} u^* v^* \ge \int_{\R^N} uv,
\end{equation}
and finally, if $u, v \in L^p(\R^N)$
\[
 \int_{\R^N} \abs{u^*-v^*}^p \le \int_{\R^N} \abs{u-v}^p.
\]
(When $p=2$, the latter inequality is a consequence of \eqref{eqCavalieri} and \eqref{eqHardyLittlewood}.) The proof of these inequalities relies essentially on the monotonicity and preservation of measure of the symmetrization of sets \cite{CroweZweibelRosenbloom1986}.

Symmetrization also has more geometrical properties, such as the P\'olya-Szeg\H o inequality: if $u \in W^{1,p}(\R^N)$ is nonnegative, then $u^* \in W^{1, p}(\R^N)$, and
\begin{equation}
\label{ineqPS}
 \int_{\R^N} \abs{\nabla u^*}^p \le \int_{\R^N} \abs{\nabla u}^p.
\end{equation}
Another geometrical is the Riesz--Sobolev rearrangement inequality: if $u \in L^p(\R^N)$, $v \in L^q(\R^N)$ and $w \in L^r(\R^N)$ are nonnegative functions and $\frac{1}{p}+\frac{1}{q}+\frac{1}{r}=2$, then
\begin{equation}
\label{ineqRS}
 \int_{\R^{N}} \int_{\R^N} u(x)v(y)w(x-y)\,dx\,dy \le \int_{\R^N}\int_{\R^N}u^*(x)v^*(y)w^*(x-y)\,dx\,dy.
\end{equation}

Polarization is a tool to study and prove geometrical properties of the rearrangement \cite{Baernstein1994, Dubinin1985, Dubinin1987, Dubinin1991, Solynin1996, Wolontis1952}.
If $H \subset \R^N$ is a closed halfspace, $\sigma_H : \R^N \to \R^N$ is the reflexion with respect to $\partial H$ and $u : \R^N \to \R$, the polarization of $u$ with respect to $H$ is the function $u^H : \R^N \to \R$ defined by
\[
  u^H=\begin{cases}
        \max \big(u(x), u (\sigma_H(x))\big) &\text{if $x \in H$},\\
        \min \big(u(x), u (\sigma_H(x))\big) &\text{if $x \not \in H$},
      \end{cases}
\]
The polarizations compares thus the values of the function on both sides of $\partial H$, and keeps the larger value in $H$. The polarization is also called two-point rearrangement.

A key point for the proof of symmetrization inequalities is that symmetrization is a limit of polarizations. Brock and Solynin have proved that for every $u \in L^p(\R^N)$, there exists a sequence of closed half-spaces $(H_n)_{n \ge 1}$ such that 
\[
  u^{H_1 H_2 \dots H_{n-1}H_n} \to u^*
\]
in $L^p(\R^N)$ as $n \to \infty$ \cite{BrockSolynin2000}. Noting that $\Norm{\nabla u^{H_1\dots H_n}}_{L^p}=\Norm{\nabla u}_{L^p}$, it follows that $u^{H_1 H_2 \dots H_{n-1}H_n} \weakto u^*$ weakly in $L^p$ and that one has the P\'olya--Szeg\H o inequality
\[
 \Norm{\nabla u}_{L^p} \le \liminf_{n \to \infty} \Norm{ \nabla u^{H_1\dots H_n}}_{L^p}=\Norm{\nabla u}_{L^p}.
\]
Whereas polarization allows to prove easily the Poly\'a--Szeg\H o inequality \eqref{ineqPS}, it does not give a proof of the Riesz--Sobolev rearrangement inequality \eqref{ineqRS}, as the integral $\int_{\R^{N}}\int_{\R^{N}} u(x)v(y) w(x-y)\,dx\,dy$ can decrease under polarization of $u$, $v$ and $w$ \cite[Corollary 4.3]{VanSchaftingen2006AIHP}. 

The convergence result of Brock and Solynin has been improved in a first way by proving that the sequence of polarizations can be chosen independently of $u \in L^p(\Omega)$, i.e., there exists a sequence of closed half-spaces $(H_n)_{n \ge 1}$ such that for every $u \in L^p(\R^N)$,
\[
 u^{H_1 H_2 \dots H_{n-1}H_n} \to u^*
\]
in $L^p(\R^N)$ as $n \to \infty$ \cite{VanSchaftingen2006PAMS}. 

While the proofs of both convergence results construct the sequence by some maximization procedure at each step, it is difficult to imagine writing down explicitely a sequence given by the proofs. This difficulty was first overcome by establishing convergence under a density condition on the sequence $(H_n)$ \cite{VanSchaftingen2006TMNA}. As a byproduct, one could obtain the convergence of random sequences of polarizations. However, the proof was indirect: the convergence relied on the fact that the sequence contained subsequences similar to the sequence obtained in the previous less explicit results.

\bigskip

The goal of this paper is to obtain an explicit sequence of closed half-spaces for which the convergence of the iterated polarizations can be proved directly and simply. We prove that if $(H_n)_{n \ge 1}$ is a dense sequence in the set of closed halfspaces of which $0$ is an interior point, if $u \in L^p(\Omega)$, and if one sets
\[
  \left\{
\begin{aligned}
u_0&=u,\\
u_{n+1}&=u_n^{H_1 \dots H_{n+1}},
\end{aligned}
  \right.
\]
i.e., $u_n$ is obtained by polarization $u$ iteratively $\frac{n(n+1)}{2}$ times, then, 
\[
 u_n \to u
\]
in $L^p(\R^N)$ as $n \to \infty$. 

There are three main steps in the proof. First, as in the previous work \cite{Baernstein1994, BrockSolynin2000, VanSchaftingen2006PAMS, VanSchaftingen2006TMNA}, one begins by remarking that the sequence $(u_n)_{n \in \N}$ is relatively compact (see lemma~\ref{lemmaCompactness}). Next, and this is the key novelty, we use the polarizations inequalities of lemma~\ref{lemmaPolarizationInequality} to prove that any cluster point satisfies $v^{H_k}=v$ for every $k \in \N$. The last step consists in concluding therefrom with classical arguments that $v=u^*$. This proof seems one of the most direct proofs of the approximation of symmetrization by polarizations.

\bigskip 

The paper is organized as follows. In section~\ref{sectionTools}, we present the tools used in the proof of our main convergence result. In section~\ref{sectionProof}, we prove the convergence result. Finally, in section~\ref{sectionExtension}, we discuss various variants and extensions of our results.

\section{Tools}
\label{sectionTools}

In this section, we study the three main ingredients of our proof. The material of this section is not new, but is presented here with some detail to give to the reader an idea of a complete self-contained proof of the convergence of our iteration scheme.

\subsection{Continuity and compactness}
First let us recall the continuity properties of symmetrization. The first lemma is a consequence of well-known inequalities for symmetrization.

\begin{lemma}
\label{lemmaSymmetrizationContinuous}
Let $(u_n)_{n \in \N}$ in $L^p(\R^N)$ be nonnegative and converge to $u \in L^p(\R^N)$. Then $u_n^* \in L^p(\R^N)$, $u^* \in L^p(\R^N)$ and  $u_n^* \to u^*$ in $L^p(\R^N)$.
\end{lemma}

\begin{proof}[Sketch of the proof]
The proof of lemma~\ref{lemmaSymmetrizationContinuous} relies first on the Cavalieri principle: one has 
\[
  \int_{\R^N} f(u^*)=\int_{\R^N} f(u)
\]
for every Borel-measurable function $f : \R^+ \to \R^+$ such that $f(0)=0$ \cite{BrockSolynin2000}.
The second ingredient is the inequality
\[
 \int_{\R^N} \varphi(u^*-v^*)\le \int_{\R^N} \varphi(u-v),
\]
for every nonnegative convex function $\varphi : \R \to \R^+$ such that $\varphi(0)=0$ \cite{CroweZweibelRosenbloom1986,CrandallTartar1980}.
\end{proof}

A similar property holds for polarization.

\begin{lemma}
\label{lemmaPolarizationContinuous}
Let $(u_n)_{n \in \N}$ in $L^p(\R^N)$ be nonnegative and converge to $u \in L^p(\R^N)$ and let $H \subset \R^N$ be a closed halfspace. Then $u_n^H \in L^p(\R^N)$, $u^H \in L^p(\R^N)$ and  $u_n^H \to u^*$ in $L^p(\R^N)$.
\end{lemma}

The proof of lemma~\ref{lemmaPolarizationContinuous} is similar to the proof of the preceding lemma~\ref{lemmaPolarizationContinuous}; it relies on similar properties of symmetrization, whose proofs are even simpler \cite{BrockSolynin2000}.

\begin{lemma}[Brock and Solynin \protect{\cite[lemma 6.1]{BrockSolynin2000}}]
\label{lemmaCompactness}
Let $u \in L^p(\R^N)$ be nonnegative, and let $(H_n)_{n \ge 1}$ be a sequence of closed halfspaces. 
If for every $n \in \N$,  $H_n \ni 0$, then the sequence $(u^{H_1\dots H_n})_{n \in \N}$ is relatively compact in $\R^N$.
\end{lemma}

\begin{proof}
The proof relies on the Riesz--Fr\'echet--Kolmogorov compactness criterion. First recall that
\[
 \int_{\R^N} \abs{u^{H_1\dots H_nH_{n+1}}}^p =\int_{\R^N} \abs{u^{H_1\dots H_n}}^p,
\]
so that the sequence is bounded in $L^p(\R^N)$. Next, for every $\epsilon > 0$, there exists $R> 0$ such that 
\[
 \int_{\R^N\setminus B(0,R)} \abs{u}^p \le \epsilon.
\]
Since $0 \in H_{n+1}$, one has
\[
 \int_{\R^N\setminus B(0,R)} \abs{u^{H_1\dots H_{n+1}}}^p
 \le \int_{\R^N\setminus B(0,R)} \abs{u^{H_1\dots H_{n+1}}}^p
\]
(see \cite{BrockSolynin2000} or lemma~\ref{lemmaPolarizationInequality} below). Therefore, for every $n \in \N$, 
\[
 \int_{\R^N\setminus B(0,R)} \abs{u}^p \le \epsilon.
\]
Finally if $(\rho_\delta)$ is a family of radial and radially decreasing mollifiers, for every $\epsilon > 0$, there exists $\delta > 0$ such that 
\[
 \int_{ \R^{2N}} \abs{u(x)-u(y)}^p \rho_\delta(x-y) \,dx\,dy\le \epsilon.
\]
One also has, since $\rho_\delta$ is radial,
\begin{multline*}
  \int_{\R^{N}} \int_{\R^N} \abs{u^{H_1\dots H_{n+1}}(x)-u^{H_1\dots H_{n+1}}(y)}^p \rho_\delta(x-y) \,dx\,dy \\
\le  \int_{\R^{N}} \int_{\R^N} \abs{u^{H_1\dots H_{n}}(x)-u^{H_1\dots H_{n}}(y)}^p \rho_\delta(x-y) \,dx\,dy
\end{multline*}
see e.g. \cite[proposition 8]{VanSchaftingenWillem2004}.
Therefore,
\[
 \int_{\R^{N}} \int_{\R^N} \abs{u^{H_1\dots H_{n}}(x)-u^{H_1\dots H_{n}}(y)}^p \rho_\delta(x-y) \,dx\,dy\le  \epsilon.
\]
By the Riesz--Fr\'echet--Kolmogorov compactness criterion, the sequence is compact.
\end{proof}

\subsection{Polarization inequality}
The crucial tool in the sequel the fact the product of a function with a radial and radially decreasing functions decreases under polarization, and equality implies invariance under polarization.

\begin{lemma}
\label{lemmaPolarizationInequality}
Let $u \in L^p (\R^N)$ be nonnegative, let $H$ be a closed halfspace and let $w \in L^q(\R^N)$ be a radial and radially nonincreasing function. If $0 \in H$, then
\[
 \int_{\R^N} u w \le \int_{\R^N} u^Hw.
\]
If moreover $0$ is in the interior of $H$ and $w$ is radially decreasing, then equality holds if and only if $u=u^H$.
\end{lemma}
\begin{proof}
First, for any $x \in H$, one has, since $0 \in H$, $\abs{\sigma_H(x)} \ge \abs{x}$. Therefore, $w(x) \ge w(\sigma_H(x))$ and
\[
  u(x)w(x)+u(\sigma_H(x))w(\sigma_H(x))\le u^H(x)w(x)+u^H(\sigma_H(x))w(\sigma_H(x)).
\]
Integrating this inequality over $H$ yields the desired inequality.

Assume now that there is equality. One has then, for almost every $x \in H$, 
\[
 u(x)w(x)+u(\sigma_H(x))w(\sigma_H(x))= u^H(x)w(x)+u^H(\sigma_H(x))w(\sigma_H(x))
\]
Since $0$ is interior to $H$, $\abs{x} < \abs{\sigma_H(x)}$. Hence, since $w$ is radially decreasing, $w(x) > w(\sigma_H(x))$. Therefore, one must have $u^H(x)=u(x)$ and $u^H(\sigma_H(x))=u(\sigma_H(x))$.
\end{proof}

\subsection{Symmetrized functions and polarization}

The last ingredient that we will use is a characterization of functions invariant under symmetrization (see \cite[lemma 6.3]{BrockSolynin2000}).

\begin{lemma}
\label{lemmaSymmetrizationPolarization}
Let $u \in L^p(\R^N)$ be nonnegative. The following are equivalent
\begin{enumerate}[i)]
\item for every closed halfspace $H \subset \R^N$ such that $0 \in H$, $u^H=u$,
\item $u^*=u$.
\end{enumerate}
\end{lemma}
\begin{proof}
The statement $u=u^*$ is equivalent to: for every almost every $x, y \in \R^N$, if $\abs{x} \le \abs{y}$, $u(x) \ge u(y)$. This is equivalent in turn to: for every closed halfspace $H$ such that $0 \in H$, for almost every $x$, $u(x)\ge u(\sigma_H(x))$. By definition of polarization, this is: for every such halfspace $H$, $u^H=u$,
\end{proof}

\subsection{Topology of halplanes}
We define $\mathcal{H}$ as the set of closed half-spaces of $\R^N$, $\mathcal{H}_*$ the set of closed half-spaces containing $0$ and $\breve{\mathcal{H}}_*$ the set of closed half-spaces of which $0$ is an interior point. One can endow $\mathcal{H}$ with a topology that ensures that $H_n \to H$ if there is a sequence of isometries $i_n : \R^N \to \R^N$ such that  $H_n=i_n(H)$ and $i_n$ converges to the identity as $n \to \infty$ (see e.g. \cite[lemma 5.2]{BrockSolynin2000} or \cite[section 2.4]{VanSchaftingen2006TMNA}).

\section{Proof of the main Theorem}
\label{sectionProof}

\begin{theorem}
\label{theoremMain}
Given $u \in L^p(\R^N)$, and a sequence $(H_n)_{n \ge 1}$ in $\breve{\mathcal{H}}_*$, define $(u_n)$ by 
\[
\left\{
\begin{aligned}
  u_0&=u, \\
  u_{n+1}&=u^{H_{n+1}\dotsc H_2 H_1}u_n.
\end{aligned}
\right.
\]
If $(H_n)_{n \ge 1}$ is dense in $\mathcal{H}_*$, then 
\[
 u_n \to u^*
\]
in $L^p(\R^N)$ as $n \to \infty$.
\end{theorem}
\begin{proof}
By lemma~\ref{lemmaCompactness}, the sequence $(u_n)_{n \ge 0}$ is compact. Let $v$ be an accumulation point and assume that $u_{n_k} \to v$ in $L^p(\R^N)$. Fix a radial radially decreasing function $w \in L^q(\R^N)$, with $\frac{1}{p}+\frac{1}{q}=1$, e.g., $w(x)=e^{-\abs{x}^2}$. By an iterative application of lemma~\ref{lemmaPolarizationInequality}, for every $k \in \N$ and $l \in \N$, such that $n_k \ge l$
\[
  \int_{\R^N} u_{n_k}^{H_1\dots H_l} w \le \int_{\R^N} u_{n_{k+1}} w.
\]
As $k \to \infty$, one has, by lemma~\ref{lemmaPolarizationContinuous}, 
\begin{equation}
\label{ineqIteratedPol}
 \int_{\R^N} v^{H_1\dots H_l} w \le \int_{\R^N} v w.
\end{equation}
In the special case when $l=1$, this implies by lemma~\ref{lemmaPolarizationInequality} that $v^{H_1} =v$. Now, assume that $v=v^{H_r}$ for $1 \le r \le l$. Then $v^{H_1\dots H_l}=v$, so that \eqref{ineqIteratedPol} becomes 
\[
 \int_{\R^N} v^{H_{l+1}} w \le \int_{\R^N} v w.
\]
By lemma~\ref{lemmaPolarizationInequality}, one has $v^{H_{l+1}}=v$.

In order to apply lemma~\ref{lemmaSymmetrizationPolarization} let us prove that $v^H=v$ for every $H \in \mathcal{H}_*$. Since $(H_n)_{n \ge 1}$ is dense in $\mathcal{H}_*$, there is a subsequence $(H_{m_k})_{k \ge 1}$ and there are isometries $i_k$ such that $i_k$ converges to the identity and $H_{m_k}=i_k(H)$. Therefore, 
\[
  v^{H_{m_k}}=v^{i_k(H)}=(v\circ i_k^{-1}) ^H\circ i_k, 
\]
so that by lemma~\ref{lemmaPolarizationContinuous}, $v^H_{m_k}\to v^H$. Since $v^{H_{m_k}}=v$, 
one has by lemma~\ref{lemmaSymmetrizationPolarization}, that $v=v^*$. By lemma~\ref{lemmaSymmetrizationContinuous}, one also has $v^*=u^*$, so that $v=u^*$. Since $v$ was an arbitrary accumulation point, this ends the proof.
\end{proof}

The argument to prove that $v=v^*$ is reminiscent of the characterization by duality of symmetric functions \cite[lemma 3.1]{VanSchaftingen2006AIHP}.

\section{Possible extensions}

\label{sectionExtension}

The method used to prove theorem~\ref{theoremMain} is quite flexible; we conclude this paper by presenting variants that can be obtained in the same fashion.

\subsection{Other functional spaces}
The conclusion of theorem~\ref{theoremMain} is the convergence of the sequence of iterated polarization in $L^p(\R^N)$. The  space $L^p(\R^N)$ can be replaced by other spaces in which lemmas~\ref{lemmaSymmetrizationContinuous},~\ref{lemmaPolarizationContinuous} and~\ref{lemmaCompactness} hold; a natural class of spaces to examine are rearrangement-invariant spaces \cite{Talenti1994}. A first example is the space of continuous functions that vanish at infinity $C_0(\R^N)$. Another example is the Orlicz space $L^\Phi(\R^N)$ \cite{KrasnoselskiiRuticki1961, RaoRen1991}: one has the convergence of theorem~\ref{theoremMain} for every nonnegative $u$ in the closure of continuous functions in $L^\Phi(\R^N)$. In general, one cannot have convergence for the whole space: for example, one does not have convergence in  $L^\infty(\R^N)$ (such a convergence would mean that a finite number of polarizations bring to the symmetrization). One can make similar statements for Lorentz spaces $L^{p, q}(\R^N)$.

In another direction, as a consequence of theorem~\ref{theoremMain}, if $u : \R^N \to \R^+$ is measurable and $\muleb{N}(\{x \in \R^N \st u(x) > \lambda\}) < \infty$ for every $\lambda > 0$, then $u_n \to u$ in measure, i.e., for every $\epsilon > 0$, 
\[
  \lim_{n \to \infty} \muleb{N}(\{x \in \R^N \st \abs{u_n(x)-u(x)} > \epsilon \})=0.
\]
Similarly, the same sequence of iterated polarizations of compact sets converges to the symmetrized compact set in the Hausdorff metric (see \cite[section 3.3]{VanSchaftingen2006TMNA}).

\subsection{Using a smaller set of halfspaces}
The density condition in theorem~\ref{theoremMain} can also be relaxed to the density of $(H_n)_{n \in \N}$ in the set of closed half-space whose boundary intersects some fixed ball $B(0, \rho)$, for some fixed $\rho > 0$. The generalization relies on a following variant of lemma~\ref{lemmaSymmetrizationPolarization}.

\begin{lemma}
Let $u \in L^p(\R^N)$ be nonnegative and let $\rho > 0$. If for every $H \in \mathcal{H}_*$ with $\partial H \cap B(0, \rho)\ne \emptyset$, one has $u^H=u$, then $u^*=u$.
\end{lemma}
\begin{proof}
Proceeding as in the proof of lemma~\ref{lemmaSymmetrizationPolarization}, one obtains that $u(x) \ge u(y)$ for almost every $x, y \in \R^N$ such that $\abs{x} \le \abs{y}$, and
\[
  \abs{\abs{x}^2-\abs{y}^2} < 2\rho \abs{x-y},
\]
the latter condition coming from the restriction that the hyperplane that reflects $x$ on $y$ should intersect the ball $B(0,\rho)$.

Now assume that $x, y \in \R^N$ and $\abs{x} < \abs{y}$. There exists $k \in \N$ such that the set of points $(z_1 \dots, z_{k-1}) \in \R^{N(k-1)}$ for which one has for $i \in \{0, \dotsc, k-1\}$, $\abs{z_i} < \abs{z_{i+1}}$ and
\[
 \abs{\abs{z_i}^2-\abs{z_{i+1}}^2} < 2\rho \abs{z_i-z_{i+1}},
\]
with the convention $z_0=x$ and $z_k=y$, has positive measure.
By the first part of the proof, for almost every $x$ and $y$, one can ensure also that $u(z_i) \ge u(z_{i+1})$ for $1 \le i \le k-1$. This proves thus that for almost every $x, y \in \R^N$, if $\abs{x} < \abs{y}$, $u(x) \ge u(y)$, which implies in turn that $u^*=u$.
\end{proof}

\subsection{Approximating other symmetrizations}
The method that has been devised above is not limited to the approximation of the Schwarz symmetrization; it is in fact adapted to all the symmetrizations that have been approximated by polarizations: rearrangements on the sphere and on the hyperbolic space \cite{Baernstein1994}, Steiner symmetrization \cite{Baernstein1994, BrockSolynin2000}, spherical cap symmetrization \cite{SmetsWillem2003}, increasing rearrangement \cite{VanSchaftingen2006PAMS}, discrete symmetrization \cite{Pruss1998}.

For the sphere and the hyperbolic space, the definitions and the proofs are adapted straightforwardly. One can note that, on the sphere, no sign restriction on $u$ is needed anymore.

The Steiner symmetrization with respect to some $k$--dimensional affine subspace $S$ is defined in such a way that for every $\lambda> 0$ and every $(N-k)$--dimensional affine subspace $L$ that is normal to $S$, $\{x \in L \st u^*(x) > \lambda\}$ is a ball of $L$ centered around $L \cap S$, and has the same $(N-k)$--dimensional measure as $\{x \in L \st u (x) > \lambda \}$ \cite[definition 2.1]{VanSchaftingen2006TMNA}. If one takes $\breve{\mathcal{H}}_*$ to be the set of closed affine subspace $H  \subset \R^N$ such that $S$ is in the interior of $H$, a variant of theorem~\ref{theoremMain} holds.

The spherical cap symmetrization, or foliated Schwarz symmetrization with respect to some $k$--dimensional closed half affine subspace $S$ is defined similarly: for every $\lambda> 0$ and every $(N-k)$--dimensional sphere $L$ which is contained in an $(N-k+1)$--dimensional affine subspace that is normal to $\partial S$ and whose center lies on $\partial S$, the sublevel set $\{x \in L \st u^*(x) > \lambda\}$ is a geodesic ball of $L$ centered around $L \cap S$ and has the same $(N-k)$--dimensional measure as $\{x \in L \st u (x) > \lambda \}$ \cite[definition 2.3]{VanSchaftingen2006TMNA}. One takes then $\breve{\mathcal{H}}_*$ to be the set of closed affine subspaces $H  \subset \R^N$ such that $\partial H \supset \partial S$, and $S \cap H$ and a variant of theorem~\ref{theoremMain} holds.

Rearrangement can also be defined on a regular tree $\mathcal{T}_q$ \cite{Pruss1998}, with $q \ge 2$, which is a countably infinite graph that is connected, does not contain any cycle and in which every vertex has exactly $q$ edges associated to it. We identify $\mathcal{T}_q$ with its set of vertices. One can define on such a tree a spiral-like ordering \cite[definition 6.1]{Pruss1998}. For example, one has $\mathcal{T}_2 \sim \mathbf{Z}$ on which one can take the ordering $0 \prec 1 \prec -1 \prec 2 \prec -2\prec \dots$.
Given a nonnegative function $u : \mathcal{T}_q \to \R$, its symmetrization $u^*$ is the nonincreasing function such that for every $\lambda > 0$, $\{ x \in \mathcal{T}_q \st u^*(x) > \lambda\}$ and $\{ x \in \mathcal{T}_q \st u(x) > \lambda\}$ have the same number of elements. Polarizations are defined in terms of isometrical involutions $i$: one takes $u^i(x)=\max(u(x), u(i(x)))$ if $x \prec i(x)$ and $u^i(x)=\min(u(x), u(i(x)))$ otherwise.
Taking as a distance between two involutions the inverse of the diameter of the largest ball centered at the origin on which they coincide and using the tools of Pruss, one can then prove the counterpart of theorem~\ref{theoremMain}.
When $q=2$, one notes that the isometrical involutions of $\mathcal{T}_2$ form a countable set and $(H_n)_{n \ge 1}$ is then an enumeration in the set of polarizations.

In all these settings, the remark about approximating with a smaller set of polarizations holds. In the special case of the rearrangement on $\mathcal{T}_2 \sim \mathbf{Z}$, one can even take two polarizations, defined by the involutions $i_1(x)=x$ and $i_2(x)=1-x$.

\subsection{Approximation by symmetrization}
In some applications, for example in a proof of the Riesz--Sobolev inequality \cite{BrascampLiebLuttinger1974}, one needs to approximate symmetrization by other symmetrizations. The approach developped in \cite{VanSchaftingen2006PAMS,VanSchaftingen2006TMNA}, to unify the presentation of approximation of polarization and approximation by symmetrization also applies here:
One defines $\mathcal{S}$ to be the set of affine subspaces and closed half affine subspaces of $\R^N$. This set can be endowed with a partial order $\prec$ defined by $S \prec T$ if $S \subset T$ and $\partial S \subset \partial T$ \cite[definition 2.19]{VanSchaftingen2006TMNA}, and it can also be endowed with a metric \cite[section 2.4]{VanSchaftingen2006TMNA}. If $u \in L^p(\R^N)$ and $S \in \mathcal{S}$, depending on the nature of $S$, $u^S$ denotes the Schwarz, Steiner or spherical cap symmetrization, or the polarization. 

Proceeding as in the proof of theorem~\ref{theoremMain}, one proves that if $(T_n)_{n \ge 1}$ is dense in $\{T \in \mathcal{S} \st S \prec T\}$, and if one sets for $u \in L^p(\R^N)$ with $u \ge 0$,
\[
  \left\{
\begin{aligned}
u_0&=u,\\
u_{n+1}&=u_n^{T_1 \dots T_{n+1}},
\end{aligned}
  \right.
\]
then, $u_n \to u$ in $L^p(\R^N)$. 
%
%
\def\cprime{$'$}
\providecommand{\bysame}{\leavevmode\hbox to3em{\hrulefill}\thinspace}
\providecommand{\MR}{\relax\ifhmode\unskip\space\fi MR }
\providecommand{\MRhref}[2]{%
  \href{http://www.ams.org/mathscinet-getitem?mr=#1}{#2}
}
\providecommand{\href}[2]{#2}

\end{document}